\documentclass[11pt]{amsart}
\usepackage{comment}
\usepackage{amssymb,amsmath,amsthm,mathrsfs,amstext,amssymb}
\usepackage{tabularx}
\usepackage{array}
\usepackage[square,sort,comma,numbers]{natbib}
\usepackage{url} 
\usepackage{hyperref}
\usepackage{accents}

\theoremstyle{plain}
\newtheorem{lemma}{Lemma}[section]
\newtheorem{prp}[lemma]{Proposition}
\newtheorem{thm}[lemma]{Theorem}
\newtheorem*{thm*}{Theorem}
\newtheorem{crl}[lemma]{Corollary}
\newtheorem*{crl*}{Corollary}
\newtheorem{question}[lemma]{Question}
\newtheorem{claim}[lemma]{Claim}
\newtheorem{dfn}[lemma]{Definition}

\newcommand{\om}{\omega}
\newcommand{\ZFC}{\mathsf{ZFC}}
\newcommand{\CH}{\mathsf{CH}}
\newcommand{\forces}{\Vdash}

\usepackage{comment}

\makeatother 

\usepackage{babel}

\title{Filters and ideal independence}

\author{J. Cancino-Manr\'iquez}
\address{Institute of Mathematics, Czech Academy of Sciences, \v{Z}itn\'a 25, Praha 1,
Praha, Czech Republic}
\email{cancino@math.cas.cz}

\author{V. Fischer}
\address{Institute of Mathematics, Univeristy of Vienna, Kolingasse 14-16, 1090 Vienna, Austria}
\email{vera.fischer@univie.ac.at}

\author{C. Bacal Switzer}
\address{Institute of Mathematics, Univeristy of Vienna, Kolingasse 14-16, 1090 Vienna, Austria}
\email{corey.bacal.switzer@univie.ac.at}

\thanks{\emph{Acknowledgments.}: The authors would like to thank the Austrian Science Fund (FWF) for the generous support through START Grant Y1012-N35. 
}



\begin{document}

\begin{abstract}
A family $\mathscr{I} \subseteq [\omega]^\omega$ such that for all finite $\{X_i\}_{i\in n}\subseteq \mathcal I$ and $A \in \mathscr{I} \setminus \{X_i\}_{i\in n}$, the set $A \setminus \bigcup_{i < n} X_i$ is infinite, is said to be ideal independent. An ideal independent family which is maximal under inclusion is said to be a maximal ideal independent family and the least cardinality of such family is denoted $\mathfrak{s}_{mm}$. 

We show that $\mathfrak{u}\leq\mathfrak{s}_{mm}$, which in particular establishes 
the independence of $\mathfrak{s}_{mm}$ and $\mathfrak{i}$. Given an arbitrary set  $C$ of uncountable cardinals, we show how to simultaneously adjoin via forcing maximal ideal independent families of  cardinality $\lambda$ for each $\lambda\in C$, thus establishing the consistency of $C\subseteq \hbox{spec}(\mathfrak{s}_{mm})$. Assuming $\CH$, we construct  a maximal ideal independent family, which remains maximal after forcing with any proper, $^\omega\omega$-bounding, $p$-point preserving forcing notion and evaluate $\mathfrak{s}_{mm}$ in several well studied forcing extensions. 
\end{abstract}

\maketitle

\section{Introduction}

Given a family $\mathscr I \subseteq [\omega]^\omega$, the {\em ideal associated to} $\mathscr I$ is the collection of all $A \subseteq \omega$ so that $A \subseteq^* \bigcup_{i < n} X_i$ for some finite $\{X_i\}_{i\in n}\subseteq \mathscr{I}$ where $\subseteq^*$ means inclusion mod finite. A family $\mathscr{I}$ is {\em ideal independent} if no $A \in \mathscr{I}$ is in the ideal generated by $\mathscr{I} \setminus \{A\}$. More precisely,  $\mathscr{I}$ is ideal independent, if whenever $X_0, ..., X_{n-1}, A \in \mathscr{I}$ and $A \subseteq^* \bigcup_{i < n} X_i$ then there is an $i < n$ so that $X_i = A$. Almost disjoint families and independent families are both examples of ideal independent families. An easy application of Zorn's lemma shows that there are maximal ideal independent families, however it is not always the case that a maximal almost disjoint or a maximal independent family is maximal ideal independent. Thus, the cardinal characteristic $\mathfrak{s}_{mm}$, defined as the least cardinality of a maximal ideal independent family, becomes of interest.

An earlier investigation of $\mathfrak{s}_{mm}$ can be found \cite{cancino_guzman_miller_2021}, where it is shown that ${\rm max}\{\mathfrak{d}, \mathfrak{r}\}\leq\mathfrak{s}_{mm}$ and that each of the following inequalities $\mathfrak{u} < \mathfrak{s}_{mm}$, $\mathfrak{s}_{mm} < \mathfrak{i}$, $\mathfrak{s}_{mm} < \mathfrak{c}$ is consistent. Here $\mathfrak{d}$, $\mathfrak{u}$, $\mathfrak{r}$,  $\mathfrak{i}$ denote the is the dominating number, the ultrafilter number, the reaping and  independence numbers, respectively.  We refer the reader to \cite{blass_handbook} for  definitions and basic properties of the combinatorial cardinal characteristics, which are not stated here. Strengthening and complimen\-ting the above results, in Section 2, we establish the following $\ZFC$ inequality, which also answers  Question 17 of \cite{cancino_guzman_miller_2021}, see Theorem \ref{mainthm1}:

\begin{thm*}
$\mathfrak{u}\leq \mathfrak{s}_{mm}$. 
\end{thm*}

As a consequence we obtain the independence of $\mathfrak{s}_{mm}$ and $\mathfrak{i}$, as the consistency of $\mathfrak{s}_{mm} < \mathfrak{i}$ is shown in \cite[Theorem 16]{cancino_guzman_miller_2021}, while the consistency of $\mathfrak{i} < \mathfrak{u}$ is established in Shelah's \cite{con_i_u} and hence by the above theorem, $\mathfrak{i} < \mathfrak{s}_{mm}$ holds in the latter model.

\begin{crl*}
$\mathfrak{s}_{mm}$ and $\mathfrak{i}$ are independent.
\end{crl*}

A key role in our investigations is taken by specific filters, which are naturally associated to a given ideal independent family. On one side, these are filters to which we refer as complemented filters, see Definition \ref{complemented_filter_def} and on the other side, filters resembling the notion of a diagonalization filter for an independent family, see for example \cite[Definition 1]{VFSS1}. In difference with earlier instances of diagonalization reals, associated to say almost disjoint families, towers, or cofinitary groups, the existence of a diagonalization
real for a given ideal independent family, employs a Cohen real (see Lemma \ref{extension_lemma}). Adjoining diagonalization reals for ideal independent families cofinally along an appropriate finite support iteration, as well as building on and modifying earlier forcing constructions used to  control for example the spectrum of independence (see in particular \cite{VFSS1, VFSS2}) we establish the following (see Theorem \ref{mainthm2}):


\begin{thm*} 
(GCH) Let $R$ be a set of regular uncountable cardinals. Then, there is a ccc generic extension in which for every $\lambda\in R$ there is a maximal ideal independent family of cardinality $\lambda$. 
\end{thm*}

Moreover, we look at the preservation of small witnesses to $\mathfrak{s}_{mm}$. The preservation of the maximality of extremal sets  of reals, like mad families, maximal eventually different families of reals, or maximal independent families, under forcing iterations is usually a non-trivial task and often involves the construction of a combinatorial object which is maximal in a strong sense, examples given by tight almost disjoint and selective independent families. Partially inspired by the notion of an $\mathcal{U}$-supported maximal independent family in the higher Baire spaces given in \cite{VFDM}, in Definition \ref{encompassing} we introduce the notion of an $\mathcal{U}$-encompassing ideal independent family and establish the following general preservation result (see Theorems \ref{mainthm3} and \ref{preserving_encompassing}). We should point out, that even though, the latter two notions have some superficial similarities, they do remain significantly different, as they reflect the structure of rather distinct combinatorial sets of reals. 

\begin{thm*}
(CH) There is a maximal ideal independent family $\mathscr{I}$ which remains maximal, and so a witness to  $\mathfrak{s}_{mm} = \aleph_1$, in any generic extension obtained by a proper, $\om^\om$-bounding, $p$-points preserving forcing notion.
\end{thm*}

In particular, the above result implies that in many well-studied forcing extensions, $\mathfrak{s}_{mm}=\max\{\mathfrak{d}, \mathfrak{u}\}$. We conclude the paper with a brief discussion of remaining open questions.

\subsection{Preliminaries}
Recall that given $f, g\in \om$ we write $f \leq^* g$ ($f$ is {\em eventually dominated by} $g$), provided there is $n\in\omega$ such that for all $m\geq n$, $f(m)\leq g(m)$. The cardinal $\mathfrak{b}$, the {\em bounding number} is the least size of a family $A \subseteq \om^\om$ so that no single $f \in \om^\om$ eventually dominates every $g \in A$. Dually, the {\em dominating number}, $\mathfrak{d}$, is the least size of a {\em dominating family}, that is a family $D \subseteq\om^\om$ so that every $f \in \om^\om$ is eventually dominated by some $g \in D$. We denote by $[\omega]^\om$ the {\em Ramsey space}, that is the Polish space of infinite subsets of $\om$. Often it is convenient to quotient this space by the ideal of finite sets. For instance, if $A, B$ in $[\om]^\om$ then we write $A \subseteq^* B$, read ``$A$ is {\em almost contained in} $B$" if $A \setminus B$ is finite. Similarly we say that $A$ and $B$ are {\em almost equal}, denoted $A =^* B$, if their symmetric difference is finite, and we say that $A$ and $B$ are {\em almost disjoint}, denoted $A \cap B =^* \emptyset$, if their intersection is finite. Given a family $\mathcal A \subseteq [\om]^\om$ we say that $\mathcal A$ has the {\em finite intersection property} if any finite subfamily has infinite intersection. If $\mathcal A$ has the finite intersection property then it may have a {\em pseudo-intersection},  i.e. a set $B\in [\om]^\om$ so that $B \subseteq^* A$ for every $A \in \mathcal A$. The cardinal characteristic $\mathfrak{p}$ is the least size of a family with the finite intersection property with no pseudo-intersection. An ultrafilter is is said to be {\em principal} if it contains a singleton and {\em non-principal} otherwise. Unless otherwise stated we will assume all ultrafilters are non-principal. If $\mathscr{U}$ is an ultrafilter then a {\em base} for $\mathscr{U}$ is a subset $\mathcal B \subseteq \mathscr{U}$ so that every element $A \in \mathscr{U}$ almost contains some $B \in \mathcal{B}$. In this case, we say that $\mathcal B$ {\em generates} $\mathscr{U}$, sometimes denoted $\langle \mathcal B \rangle$. The cardinal $\mathfrak{u}$, the {\em ultrafilter number}, is the least size of a non-principal ultrafilter base. 
For an ultrafilter  $\mathscr{U}$ we say that
\begin{enumerate}
    \item $\mathscr{U}$ is a {\em $p$-point} if every countable subfamily of $\mathscr{U}$ has a pseudo-intersection in $\mathscr{U}$;
    \item $\mathscr{U}$ is a {\em q-point} if every partition of $\om$ into finite sets $\{I_n\}_{n < \om}$ there is a $U \in \mathscr{U}$ so that $|U \cap I_n| = 1$ for each $n < \om$;
    \item $\mathscr{U}$ is {\em Ramsey}, or, {\em selective} if it is a $p$-point and a q-point.
    \item $\mathscr{U}$ is a $p_{\mathfrak{c}}$-point if any $\mathscr{F}\subseteq\mathscr{U}$, $|\mathscr{F}|<\mathfrak{c}$ has a pseudo-intersection in $\mathscr{U}$.
\end{enumerate}

Finally, a family $\mathcal I \subseteq [\om]^\om$ is said to be {\em independent} if whenever $\mathcal A, \mathcal B$ are finite, disjoint, non-empty subfamilies of $\mathcal I$, the set $\bigcap A \setminus \bigcup \mathcal B$ is infinite. The least size of a maximal independent family is denoted $\mathfrak{i}$.

\section{Ultrafilters and ideal independence}

A central role in our proof of $\mathfrak{u}\leq\mathfrak{s}_{mm}$ is played by the following filters:

\begin{dfn}\label{complemented_filter_def}
Let $\mathscr{I}$ be an ideal independent family. For any $A\in\mathscr{I}$, let $\mathcal{F}(\mathscr{I},A)$ be the filter on generated by the family $\{A\setminus\bigcup F:F\in [\mathscr{I}]^{<\omega}\land A\notin F\}$.  We refer to the filters of the form $\mathcal{F}(\mathscr{I}, A)$ as the {\em complemented filters of} $\mathscr{I}$, while for a fixed $A \in \mathscr{I}$ we say that  $\mathcal{F}(\mathscr{I}, A)$ is the {\em complemented filter (of} $\mathscr{I}$) {\em corresponding to} $A$.
\end{dfn}

Note that an ideal independent family $\mathscr{I}$ is maximal if and only if every $X \in [\omega]^\omega$ is either in the ideal generated by $\mathscr{I}$ or belongs to at least one of the filters $\mathcal{F}(\mathscr{I}, A)$ (these two possibilities are not mutually exclusive). The name ``complemented" comes from this observation: under maximality every element of the complement of the ideal generated by $\mathscr{I}$ is in some completemented filter.

\begin{thm}\label{mainthm1}
$\mathfrak{u}\leq\mathfrak{s}_{mm}$.
\end{thm}

\begin{proof}
Assume otherwise $\mathfrak{s}_{mm}<\mathfrak{u}$ and let $\mathscr{I}$ be a maximal ideal independent family of minimal cardinality. Maximality implies that there is $\{ A_n\}_{n\in\omega}\subseteq\mathscr{I}$ whose union is almost equal to $\omega$. Define $B_0=A_0$ and for $n>0$, $B_n=A_n\setminus\bigcup_{i<n}A_n$. For each $n\in\omega$, let $\mathcal{F}_n$ be the filter $\mathcal{F}(\mathscr{I},A_n)\upharpoonright B_n$. Since $\mathcal{F}_n$ is not an ultrafilter, for any $s\in 2^{n}$ there is $D_s\in\mathcal{F}_n^+\setminus \mathcal{F}_n$ such that for different $s,r\in 2^n$, $D_r\cap D_s=\emptyset$. Now, for each $f\in 2^\omega$, define $D^f=\bigcup_{n\in\omega}D_{f\upharpoonright n}$. Note that $D^f$ is not almost contained in the union of finitely many elements from $\mathscr{I}$. Indeed, for any finite $F\subseteq\mathscr{I}$, let $n\in\omega$ be such that $A_n\notin F$. We can assume that $\{A_i:i<n\}\subseteq F$. Since $D_{f\upharpoonright n}$ is $\mathcal{F}_n$-positive, it can not be covered by $\bigcup F$. It follows that $D^f$ can not be covered by $\bigcup F$ either. By maximality of $\mathscr{I}$, for any $f\in 2^\omega$, there are $A_f\in\mathscr{I}$ and $F_f\in[\mathscr{I}\setminus\{A_f\}]^{<\omega}$ such that $A_f\setminus \bigcup F_f\subseteq^* D^f$. Note that for no $n\in\omega$ it is the case that $A_{n}=A_f$: otherwise, for some $n\in\omega$ we would have $A_{n}\setminus F_f\subseteq^* D^f$, which implies $B_{n}\setminus \bigcup F_f=A_{n}\setminus\left(\bigcup F_f\cup\bigcup_{i< n}A_i\right)\subseteq^* D_{f\upharpoonright n}$, contradicting the choice of the set $D_{f\upharpoonright n}$. Since $\mathfrak{s}_{mm}<\mathfrak{c}$, there are different $f,g\in 2^\omega$ such that $A_f=A_g$ and $F_f=F_g$. By construction, we have that $A_f\setminus F_f\subseteq^* D^f\cap D^g\subseteq\bigcup_{i\leq n_0} B_i$, where $n_0\in\omega$ is the maximal natural number such that $f\upharpoonright n_0=g\upharpoonright n_0$.  But $\bigcup_{i\leq n_0}B_i=\bigcup_{i\leq n_0}A_{i}$, which means that $A_f\setminus \bigcup F_f\subseteq^* \bigcup_{i\leq n_0}A_{n_0}$, and so $A_f\subseteq^*\bigcup F\cup\bigcup_{i\leq n_0}A_i$, a contradiction.
\end{proof}

The above proof shows that whenever $\mathscr{I}$ is a maximal ideal independent family such that  $|\mathscr{I}|<\mathfrak{c}$, then there are at most finitely many $A\in\mathscr{I}$ for which the corresponding complemented filter is not an ultrafilter. Indeed, if there were infinitely many such $A$'s, one could add them to the family $\{A_n\}_{n\in\omega}$ in the above proof and proceed along the same lines to reach a contradiction. However, this is not necessarily the case for ideal independent families with cardinality $\mathfrak{c}$, as for example any completely separable maximal almost disjoint  family $\mathcal{A}$ is a maximal ideal independent family \footnote{Recall that a maximal almost disjoint family $\mathcal{A}$ is completely separable if for any $X\in[\omega]^\omega$, there is $B\in\mathcal{A}$ such that $B\subseteq X$ or $X$ belongs to the ideal generated by $\mathcal{A}$.} and for any $A\in\mathcal{A}$, the corresponding complemented filter $\mathcal{F}(\mathscr{I},A)$ is the collection of cofinite subsets in $A$. It remains of interest to characterise those ideal independent families  $\mathscr{I}$ for which there is $A \in \mathscr{I}$ such that  $\mathcal{F}(\mathscr{I}, A)$ is an ultrafilter. 
The following two questions remain open:

\begin{question}
Is it consistent that there are no maximal ideal independent families $\mathscr{I}$ with the property that $\mathcal{F}(\mathscr{I}, A)$ is an ultrafilter for some $A \in \mathscr{I}$ ?
\end{question}

As pointed out above, a positive answer to this question implies $\mathfrak{s}_{mm} = 2^{\aleph_0}$.

\begin{question}
Is it consistent that for every maximal ideal independent family $\mathscr{I}$ there is an $A\in\mathscr{I}$ so that $\mathcal{F}(\mathscr{I}, A)$ is an ultrafilter?
\end{question}

A positive answer to this question would imply that there are no completely separable maximal almost disjoint families. It is known that such families exist under either $\mathfrak{s} \leq \mathfrak{a}$ or $2^{\aleph_0} < \aleph_\omega$, see \cite{heike-dilip-juris} and \cite{hrusak-separable-mad}, respectively. 
In contrast, we show that at least under certain assumptions there are maximal ideal independent families $\mathscr{I}$ so that {\em every} complemented filter of $\mathscr{I}$ is an ultrafilter:

\begin{prp}
If $\mathfrak{p} = \mathfrak{c}$ then there is a maximal ideal independent family $\mathscr{I}$ with the property that for every $A \in \mathscr{I}$ the corresponding complemented filter $\mathcal{F}(\mathscr{I}, A)$ is a $p_{\mathfrak{c}}$-point.
\end{prp}


\begin{proof}
Observe that if $\mathfrak{p} = \mathfrak{c}$, $\mathcal{Z}\subseteq [\om]^\om$, $|\mathcal{Z}| < \mathfrak{c}$ and $X$ is not in the ideal generated by $\mathcal{Z}$, then there is a $Y \in [X]^\omega$ which is almost disjoint from every element of $\mathcal{Z}$. Indeed, consider the forcing notion $\mathbb{P}_{X, \mathcal Z}$ consisting of pairs $(s, F)\in [X]^{<\omega}\times[\mathcal{Z}]^{<\omega}$, where $(s_1, F_1) \leq (s_0, F_0)$ if $s_1 \supseteq s_0$, $F_1 \supseteq F_0$ and 
$(s_1 \setminus s_0)\cap(\bigcup F_0)=\emptyset$. It is easy to see that $\mathbb{P}_{X, \mathcal Z}$ is $\sigma$-centered and adds a subset of $X$ which is almost disjoint from every element of $\mathcal{Z}$. Since $\hbox{MA}(\sigma\hbox{-centered})$ holds (by $\mathfrak{p} = \mathfrak{c}$, see \cite{blass_handbook}), the claim follows.

The family $\mathscr{I}$ is defined recursively. Fix an enumeration of $[\om]^\om$ indexed by the even ordinals, $\{X_\alpha \; | \; \alpha < \mathfrak{c}, \; {\rm even}\}$ so that every element of $[\om]^\om$ appears unboundedly often and an enumeration of $2^{\aleph_0}$ indexed by the odd ordinals $\{\beta_\alpha, \; | \; \alpha < \mathfrak{c} \; {\rm odd}\}$ so that every ordinal appears unboundedly often, $\beta_\alpha \leq \alpha$ for all $\alpha$, and for every pair $(X, \gamma) \in [\omega]^\omega \times \mathfrak{c}$ there is an $\alpha$ such that $(X, \gamma) = (X_\alpha, \beta_{\alpha+1})$. Let $\mathscr{I}_\omega = \{A_n\}_{n\in\omega}$ be a partition of $\om$ into infinitely many infinite sets. Suppose $\mathscr{I}_\alpha = \{A_\gamma\}_{\gamma < \alpha}$ has been defined for some even $\alpha < \mathfrak{c}$ (limits are even) and  $\mathscr{I}_\alpha$ is ideal independent. First we define a set $A_\alpha$. If $X_\alpha$ is not in the ideal generated by $\mathscr{I}_\alpha$ then by the above observation, there is $A_\alpha \subseteq X_\alpha$ which is almost disjoint from every $A_\gamma$, for $\gamma < \alpha$. Otherwise take $A_\alpha$ be an arbitrary set such that $\mathscr{I}_\alpha \cup \{A_\alpha\}$ is ideal independent. Next, take $\mathscr{I}_{\alpha+1} = \mathscr{I}_\alpha \cup \{A_\alpha\}$ and define $A_{\alpha+1}$ as follows: Let $\mathcal F_{\alpha+1}$ be the filter $\mathcal F(\mathscr{I}_{\alpha+1}, A_{\beta_{\alpha + 1}})$. At least one of the sets $X_\alpha$, $\omega \setminus X_\alpha$ is $\mathcal F_{\alpha+1}$-positive. Let $Y_\alpha$ be this one (if they both are, choose $X_\alpha$). Since $\mathfrak{p} = \mathfrak{c}$, there are  a pseudo-intersection   $Y^*_\alpha$ of $\mathcal{F}_{\alpha+1} \cup \{Y_\alpha\}$ and 
a set $Y^{**}_\alpha$ which is almost disjoint from every element of $\mathscr{I}_{\alpha+1} \cup \{Y^*_\alpha\}$. Finally, take $A_{\alpha+1} = (A_{\beta_{\alpha+1}} \setminus Y^*_\alpha) \cup Y^{**}_\alpha$ and let $\mathscr{I}_{\alpha+2} = \mathscr{I}_{\alpha+1} \cup \{A_{\alpha+1}\}$. This completes the construction.

Let $\mathscr{I} = \bigcup_{\alpha < \mathfrak{c}} \mathscr{I}_\alpha$. To see that $\mathscr{I}$ is ideal independent, it suffices to show that  $\mathscr{I}_{\alpha+1} \cup \{A_{\alpha+1}\}$ is ideal independent for $\alpha$ even. To see this, note that $A_{\alpha+1}$ is not in the ideal generated by $\mathscr{I}_{\alpha+1}$, since it contains a set almost disjoint from every element of $\mathscr{I}_{\alpha+1}$. Suppose now that there are finite $F \subseteq \mathscr{I}_{\alpha+1}$ and $A \in \mathscr{I}_{\alpha+1}$ such that $A \setminus \bigcup F \subseteq^* A_{\alpha+1}$. Since $Y^{**}_\alpha$ is almost disjoint from every element of the family  $\mathscr{I}_{\alpha+1}$, it follows that $A \setminus \bigcup F \subseteq^* A_{\beta_{\alpha+1}} \setminus Y^{*}_\alpha \subseteq^* A_{\beta_{\alpha + 1}}$ and thus by ideal independence of $\mathscr{I}_{\alpha+1}$ we obtain $A = A_{\beta_{\alpha+1}}$. However, this implies $Y^*_\alpha \subseteq^* A_{\beta_{\alpha+1}}\setminus \bigcup F \subseteq^* A_{\beta_{\alpha+1}} \setminus Y^*_\alpha$ which is a contradiction. 
To see that $\mathscr{I}$ is maximal, note that if $X$ is not in the ideal generated by $\mathscr{I}$ then by the even part of the construction there is $Y \in \mathscr{I}$, such that $Y \subseteq X$ and hence $\mathscr{I} \cup \{X\}$ is not ideal independent.

Finally, we show that $\mathcal{F}(\mathscr{I}, A)$ is a $p_{\mathfrak{c}}$-point. For each odd $\alpha$, say $\alpha=\alpha_0+1$,  $A_{\beta_\alpha} \setminus A_{\alpha} =^* A_{\beta_\alpha} \setminus (A_{\beta_\alpha} \setminus Y^*_{\alpha_0})  = Y^*_{\alpha_0}$ and so $Y^*_{\alpha_0}$ is in $\mathcal F(\mathscr{I}, A_{\beta_\alpha})$. Since every $\beta_\alpha$ is enumerated unboundedly often, pseudo-intersections are added unboundedly often along the construction and  so $\mathcal F(\mathscr{I}, A_{\beta_\alpha})$ is a $p_{\mathfrak{c}}$-point, provided it is an ultrafilter. It remains to observe that $Y^*_{\alpha_0}$ is a pseudo-intersection of either $X_\alpha$ or $\omega\setminus X_\alpha$ and so one of those sets is in the filter.
\end{proof}

\section{Arbitrarily Large  Maximal Ideal Independent Families}

In this section we examine the question of how to adjoin via forcing maximal ideal independent family of arbitrary size and thus begin an investigation of the spectrum of such families. The {\em spectrum} of maximal ideal independent families, denoted ${\rm spec}(\mathfrak{s}_{mm})$, is defined as the set of all cardinalities of maximal ideal independent families. Throughout $V$ denotes the ground model and $\mathbb{C}$ denotes the poset for adding a single Cohen real.

\begin{lemma}\label{extension_lemma}
Let $\mathscr{I}$ be an ideal independent family. There is a $ccc$ forcing $\mathbb{P}(\mathscr{I})$ which adds a set $z$ such that in $V^{\mathbb{P}(\mathscr{I})}$:
\begin{enumerate}
\item $\mathscr{I}\cup\{z\}$ is an ideal independent family, and 
\item for each $y\in V\cap ([\omega]^\omega\backslash\mathscr{I})$ the family $\mathscr{I}\cup\{z,y\}$ is not ideal independent.
\end{enumerate}
\end{lemma}

\begin{proof}
Add a Cohen real to $V$ and consider a filter $\mathcal{F}$ which contains the Cohen real and is maximal with respect to the following property: for any $X\in\mathcal{F}$, any $A\in\mathscr{I}$ and finite $F\subseteq \mathscr{I}\setminus\{A\}$, $X\cap (A\setminus\bigcup F)$ is infinite.  Let  $\mathbb{M}(\mathcal{F})$ be Mathias forcing relativized to $\mathcal{F}$, let $x$ be the generic real added by $\mathbb{M}(\mathcal{F})$ over $V^{\mathbb{C}}$, let  $\dot{x}$ be $\mathbb{M}(\mathcal{F})$-name for $x$ (in $V^{\mathbb{C}}$) and let  $\mathbb{P}(\mathscr{I})=\mathbb{C}*\mathbb{M}(\dot{\mathcal{F}})$. 


\begin{claim}
In $V^{\mathbb{P}(\mathscr{I})}$ the family $\mathscr{I}\cup\{\omega\setminus x\}$ is ideal independent.
\end{claim}
\begin{proof}
Let $F$ be a finite subset of $\mathscr{I}$. First we prove that $\bigcup F$ does not almost contain $\omega\setminus x$.  
Let $A\in\mathscr{I}\backslash F$, $(s,B)\in\mathbb{M}(\mathcal{F})$ and let $n\in\omega$ be arbitrary. Since the Cohen real belongs to $\mathcal{F}$, we can assume that $B$ is a subset of it, so $A\setminus(B\cup \bigcup F)$ is infinite. Let $k\in\omega$ be big enough so $[\max(s),k)\cap (A\setminus(B\cup \bigcup F))$ has more than $n$ elements. Then $(s\cup\{k\}, B)$ forces that $(\omega\setminus\dot{x})\setminus\bigcup F$ has more than $n$ elements and since $n$ was arbitrary, it follows that $(\omega\setminus x)\setminus\bigcup F$ is infinite. A genericity argument shows that $x\cap (A\setminus\bigcup F)$ is infinite for any $A\in\mathscr{I}$ and $F\in[\mathscr{I}\setminus\{A\}]^{<\omega}$, which implies that $(A\setminus\bigcup F)\setminus(\omega\setminus x)$ is infinite.
\end{proof}

\begin{claim}
Let $A\in ([\omega]^\omega\cap V)\backslash\mathscr{I}$. Then in $V^{\mathbb{P}(\mathscr{I})}$, $\mathscr{I}\cup\{\omega\setminus x, A\}$ is not ideal independent. 
\end{claim}

\begin{proof}
Let $A\in[\omega]^\omega\cap V$ be an arbitrary set. If there are $X\in\mathcal{F}$, $B\in\mathscr{I}$ and $F\in[\mathscr{I}\setminus\{B\}]^{<\omega}$ such that $X\cap(B\setminus\bigcup F)\subseteq^* A$, then $ x\cap(B\setminus\bigcup F)\subseteq^* A$. But $ x\cap(B\setminus\bigcup F)=(B\setminus\bigcup F)\setminus(\omega\setminus x)$, so $A$ can not be added to $\mathscr{I}\cup\{\omega\setminus x\}$, as witnessed by $B,F$ and $\omega\setminus x$. On the other hand, if for all $X\in\mathcal{F}$, $B\in\mathscr{I}$ and $F\in[\mathscr{I}\setminus\{B\}]^{<\omega}$ it happens that $X\cap(B\setminus \bigcup F)\nsubseteq^* A$, then $(\omega\setminus A)\cap X\cap(B\setminus \bigcup F)$ is infinite. Thus, by maximality of $\mathcal{F}$, $\omega\setminus A\in\mathcal{F}$, which implies that $x\subseteq^*\omega\setminus A$ and so $A\subseteq^*\omega\setminus x$. 
\end{proof}
This completes the proof of the Lemma.
\end{proof}


\begin{thm}\label{mainthm2}
Assume $GCH$. Let $C$ be a set of uncountable cardinals and let $\kappa$ be a regular uncountable cardinal such that $\sup C\leq \kappa$. Then there is a $ccc$ generic extension in which $$C\subseteq{\rm spec}(\mathfrak{s}_{mm}).$$
\end{thm}

\begin{proof}
Add $\kappa$ Cohen reals to the ground model $V$ to obtain a model of  $\mathfrak{c}=\kappa$ and for each $\lambda\in C$ let $\mathscr{I}_\lambda$ be an ideal independent family of cardinality $\lambda$. Let $\langle\lambda_\beta:\beta<\gamma\rangle$ be an enumeration of $C$. Proceed with a finite support iteration $\langle\mathbb{P}_\alpha,\dot{\mathbb{Q}}_\alpha:\alpha<\omega_1\rangle$ where each iterand is a finite support iteration of length, the cardinality of $C$, as follows:

Let $\mathbb{P}_0$ be the finite support iteration $\langle\mathbb{R}_\beta^0,\dot{\mathbb{S}}_\beta^0:\beta<\gamma\rangle$ defined by $\mathbb{R}_0^0=\mathbb{P}(\mathscr{I}_{\lambda_0})$ and $\mathbb{R}_\beta^0\Vdash\dot{\mathbb{S}}_\beta^0=\mathbb{P}(\mathscr{I}_{\lambda_\beta})$. After forcing with $\mathbb{P}_0$, for each $\beta<\gamma$, define $\mathscr{I}_\beta^0=\mathscr{I}_{\lambda_\beta}\cup\{x_\beta\}$, where $x_\beta$ is the real from Lemma \ref{extension_lemma} added by the $\beta$-step of the iteration $\mathbb{P}_0$.
Now, assume $\mathbb{P}_\beta$ and $\{\mathscr{I}_\alpha^\beta:\alpha<\gamma\}$ are defined. The next step $\dot{\mathbb{Q}}_\beta$ is the finite support iteration $\langle \mathbb{R}_\alpha^\beta,\dot{\mathbb{S}}_\alpha^\beta:\alpha<\gamma\rangle$ such that $\mathbb{R}_0^\beta=\mathbb{P}(\mathscr{I}_{0}^\beta)$ and $\mathbb{R}_\alpha^\beta\Vdash\dot{\mathbb{S}}_\alpha^\beta=
\mathbb{P}(\mathscr{I}_{\alpha}^\beta)$. In $V[G_{\beta+1}]$, after forcing with $\mathbb{P}_{\beta}*\dot{\mathbb{Q}}_\beta$, define $\mathscr{I}_{\alpha}^{\beta+1}=\mathscr{I}_{\alpha}^{\beta}\cup\{x_\beta\}$, where $x_\alpha$ is the real from Lemma \ref{extension_lemma} added by the $\alpha$-step of the iteration $\dot{\mathbb{Q}}_\beta$. If $\beta$ is a limit ordinal and $\mathbb{P}_\alpha$, $\mathscr{I}_\eta^\alpha$ are defined for all $\alpha<\beta$ and $\eta<\gamma$, let $\mathbb{P}_\beta$ be the finite support iteration $\langle\mathbb{P}_\alpha,\dot{\mathbb{Q}}_\alpha:\alpha<\beta\rangle$ and for $\eta<\gamma$, $\mathscr{I}_\eta^\beta=\bigcup_{\alpha<\beta}\mathscr{I}_\eta^\alpha$.

Let $\mathbb{P}_{\omega_1}$ be the above iteration and for any $\beta<\gamma$, let $\mathscr{J}_\beta=\bigcup_{\alpha<\omega_1}\mathscr{I}_\beta^\alpha$. Since any real $y$ in $V[G_{\omega_1}]$ is added in some intermediate extension, we have that $y$ is an element of $\mathscr{J}_\beta$ or it can not be added to $\mathscr{J}_\beta$. Then $\mathscr{J}_\beta$ is a maximal ideal independent family, and since $\mathbb{P}_{\omega_1}$ preserves all cardinals and we only added $\omega_1$ sets to the family $\mathscr{I}_\beta$ to obtain $\mathscr{J}_\beta$, $\mathscr{J}_\beta$ has size $\lambda_\beta$.
\end{proof}

The cardinality of a maximal ideal inde\-pen\-dent family can have countable cofinality, while the character of any ultrafilter is uncountable. The first assertion follows from the previous theorem by taking $\kappa>\aleph_\omega$. The second assertion can be found in \cite{ultrafilters_on_omega_cardinal_characteristics}.

\section{Forcing Invariant Maximal Ideal Independent Families}

In the following, we construct a maximal ideal independent family with strong combinatorial properties, which guarantee that its maximality is preserved by a large number of forcing notions.

\begin{dfn}\label{encompassing}
Let $\mathscr U$ be an ultrafilter. A maximal ideal independent family $\mathscr{I}$ is called $\mathscr{U}\hbox{-\em encompassing}$ if the following conditions hold:
\begin{enumerate}
    \item $\mathscr{U} \cap \mathscr{I} = \emptyset$, i.e. $\mathscr{I}$ is contained in the dual ideal of $\mathscr{U}$.
    \item For every $X \in \mathscr{U}$ the set of $A \in \mathscr{I}$ so that $X \in \mathcal{F}(\mathscr{I}, A)$ is co-countable. 
\end{enumerate}
\end{dfn}

\begin{thm}\label{mainthm3}
Assume $\mathsf{CH}$. For any $p$-point $\mathscr{U}$ there is a $\mathscr{U}$-encompassing maximal ideal independent family $\mathscr{I}$ such that for all $A\in\mathscr{I}$, the corresponding completemented filter $\mathcal{F}(\mathscr{I},A)$ is a $p$-point.
\end{thm}

\begin{proof}
Let $\mathscr{U}$ be a $p$-point and $\langle Y_\alpha:\alpha\in\omega_1\rangle$ be an $\subseteq^*$-decreasing sequence which generates the $p$-point $\mathscr{U}$. Let $\langle X_\alpha:\alpha\in\omega_1\rangle$ be an enumeration of all the infinite subsets of $\omega$. By recursion we construct a sequence $\langle\mathscr{I}_{\alpha}:\alpha\in[\omega,\omega_1)\rangle$ such that:
\begin{enumerate}
    \item For all $\alpha$, $\mathscr{I}_\alpha\subseteq\mathscr{U}^*$ is a countable ideal independent family.
    \item For all $\alpha$, if $X_\alpha \notin \mathscr{I}_{\alpha+1}$ then $\mathscr{I}_{\alpha+1}\cup\{X_\alpha\}$ is not an ideal independent family.
    \item For all $\alpha$, $\mathscr{I}_{\alpha+1}=\mathscr{I}_\alpha$ or $\mathscr{I}_{\alpha+1}=\mathscr{I}_\alpha\cup\{A_0^\alpha,A_1^\alpha\}$ for some $A_0^\alpha,A_1^\alpha\in\mathscr{U}^*$ and such that $A_0^\alpha\setminus A_1^\alpha,A_1^\alpha\setminus A_0^\alpha\subseteq Y_\alpha$.
    \item If $\alpha$ is a limit ordinal, then $\mathscr{I}_\alpha=\bigcup_{\beta<\alpha}\mathscr{I}_\beta$.
    \item If $A_i^\alpha\in\mathscr{I}$ is added in step $\alpha$ of the iteration, $\langle \mathcal{P}^{\alpha,i}_\beta:\beta\in[\alpha+1,\omega_1)\rangle$ is the enumeration of all partitions of $A^\alpha_i$, for $i\in 2$, and for all $\alpha$ and $\beta>\alpha$, there are a finite $F\subseteq[\beta\setminus\{\alpha\}]^{<\omega}$ and $k\in\omega$, such that for any $i,j\in 2$, $A^\alpha_i\setminus\left(A_j^\beta\cup \bigcup F\right)\setminus k$, either, is a partial selector of partition $\mathcal{P}^{\alpha,i}_\beta$, or is contained in one element of the partition $\mathcal{P}^{\alpha,i}_\beta$.
\end{enumerate}
After the recursion we define $\mathscr{I}=\bigcup_{\alpha<\omega_1}\mathscr{I}_\alpha$. Condition (1) makes sure that $\mathscr{I}$ is an ideal independent family and (2) makes sure that $\mathscr{I}$ is maximal. Condition (3) makes sure that $\mathscr{I}$ is $\mathscr{U}$-encompassing. Condition (5) makes sure that the filters $\mathcal{F}(\mathscr{I},A)$ are selective ultrafilters for all $A\in\mathscr{I}$. We start by setting $\mathscr{I}_\omega=\langle A_n:n\in\omega\rangle$ be a partition of $\omega$ into infinitely many infinite sets, $A_0^n=A_1^n=A_n$ and $\langle \mathcal{P}_\beta^{n,i}:\beta\in[\omega,\omega_1)\rangle$ the enumeration of all partitions of $A_n$. Assume $\mathscr{I}_\alpha$ has been constructed. We take care of the set $X_\alpha$. If $X_\alpha\in\mathscr{U}$, we just define $\mathscr{I}_{\alpha+1}=\mathscr{I}_\alpha$, and condition (2) from Definition \ref{encompassing} will make sure that $X_\alpha$ can not be added to the family $\mathscr{I}$. If $X_\alpha$ is in the ideal generated by the family $\mathscr{I}_\alpha$ we have nothing to do and we can define $\mathscr{I}_{\alpha+1}=\mathscr{I}_\alpha$ again. Otherwise, $X_\alpha\notin\mathscr{U}$ and $X_\alpha$ is positive relative to the ideal generated by $\mathscr{I}_\alpha$. Let $e_\alpha:\omega\to\mathscr{I}_\alpha$ be an enumeration of the elements of $\mathscr{I}_\alpha$, and define $C_0=e_\alpha(0)$, $C_{n+1}=e_\alpha(n+1)\setminus\bigcup_{i\leq n}e_\alpha(i)$. Now, if there are $n\in\omega$ and finite $F\subseteq\mathscr{I}_\alpha\setminus \{e_\alpha(0),\ldots,e_\alpha(n)\}$ such that $C_n\setminus X_\alpha\subseteq^*\bigcup F$ and $C_n\cap X_\alpha$ is infinite, then we have $C_n\setminus \bigcup F\subseteq^*C_n\cap X_\alpha\subseteq X_\alpha$, and we can define again $\mathscr{I}_{\alpha+1}=\mathscr{I}_\alpha$. So let us assume that for all $n\in\omega$, $C_n\setminus X_\alpha$ is finite or it is not covered by any $F\subseteq\mathscr{I}_\alpha\setminus \{e_\alpha(0),\dots,e_\alpha(n)\}$. Since $e_\alpha(n)$ is not almost contained in the union of finitely many elements from $\mathscr{I}_\alpha\setminus\{e_\alpha(n)\}$, and $\mathscr{I}_\alpha\setminus\{e_\alpha(n)\}$ is countable, by recursion we can construct an infinite set $B_n\subseteq e_\alpha(n)$, such that for all $Z\in\mathscr{I}_\alpha$ different from $e_\alpha(n)$, we have $Z\cap B_n=^*\emptyset$, and moreover, by going to a subset if necessary, $B_n$ is a partial selector of the partition $\mathcal{P}^{\gamma,i}_\alpha$ or is completely contained in one element of the partition $\mathcal{P}^{\gamma,i}_\alpha$, where $\gamma$ and $i$ are such that $A_i^\gamma=e_\alpha(n)$. Also, since for all $n\in\omega$, $C_n\notin \mathscr{U}$ and $\mathscr{U}$ is a $p$-point,  there is $A\in\mathscr{U}$ such that for all $n\in\omega$, $C_n\cap A$ is finite, $X_\alpha\cap A=\emptyset$ and $A\subseteq Y_\alpha$. Let $W_0,W_1$ be infinite disjoint subsets of $A$ which are not in the ultrafilter $\mathscr{U}$. Define $W$ as,
\begin{equation*}
    W=\left(\bigcup_{n\in\omega} C_n\setminus B_n\right)\setminus A
\end{equation*}
Finally, define $A_i^\alpha=W\cup W_i$, for $i\in 2$, and $\mathscr{I}_{\alpha+1}=\mathscr{I}_\alpha\cup\{A_0^\alpha,A_1^\alpha\}$. Note that $X_\alpha\subseteq A_i^\alpha$, $A_i^\alpha\setminus A_{1-i}^\alpha=W_i\subseteq Y_\alpha$. So we only have to prove that $\mathscr{I}_{\alpha+1}$ is ideal independent, but this is clear from the construction of $A_0^\alpha$ and $A_1^\alpha$.
\end{proof}

\begin{thm}\label{preserving_encompassing}
Let $\mathscr{U}$ be a $p$-point and let $\mathbb P$ be a proper, $\om^\om$-bounding forcing notion which preserves $p$-points. Then $\mathbb P$ preserves the maximality of any $\mathscr{U}$-encompassing maximal ideal independent family $\mathscr{I}$ such that for all $A\in\mathscr{I}$, the corresponding complemented filter $\mathcal{F}(\mathscr{I},A)$ is a $p$-point.
\end{thm}

Note that this theorem implies that under $\mathsf{CH}$, in the generic extension by any proper $\om^\om$-bounding $p$-point preserving forcing notion $\mathfrak{s}_{mm}$ is $\aleph_1$.

\begin{proof}
Fix an ultrafilter $\mathscr{U}$, a $\mathscr{U}$-encompassing maximal ideal independent family $\mathscr{I}$ with the property that all of the complemented filters of $\mathscr{I}$ are $p$-points, and a proper, $\om^\om$-bounding, $p$-point preserving forcing notion $\mathbb P$. Let $p \in \mathbb P$ and let $\dot{X}$ be a name so that $p \forces \dot{X} \in [\omega]^\omega$. We need to show that some $q \leq p$ forces that $\dot{X}$ cannot be added to $\mathscr{I}$ without destroying ideal independence. More precisely, this means that we need to either find a $q \leq p$ so that $q$ forces that $\dot{X}$ is in the ideal generated by $\mathscr{I}$ or else find $q \leq p$ and an $A \in \mathscr{I}$ so that $q$ forces that $\dot{X}$ is in the complemented filter corresponding to $A$. 

Thus suppose towards a contradiction that $p$ forces that $\dot{X}$ is neither in the ideal generated by $\mathscr{I}$ nor in any filter $\mathcal{F}(\mathscr{I}, A)$ for any $A \in \mathscr{I}$. Note that this implies in particular that $\dot{X}$ is not in $\mathscr{U}$ since if it were, then in would be in some filter $\mathcal{F}(\mathscr{I}, A)$ (in fact co-countably many). Since $\mathbb P$ preserves $\mathscr{U}$ it follows that $p$ forces that the complement of $\dot{X}$ is in $\mathscr{U}$ and therefore we can find a $q \leq p$ and a $Z \in \mathscr{U}$ so that $q \forces \dot{X} \cap Z = \emptyset$. Fix such a $q$ and $Z$. To complete the proof it suffices to therefore show that some $r \leq q$ forces that $n \in \dot{X}$ for some $n \in Z$. 

For any $u \in \mathbb P$ let $X_u = \{n\in \omega \; | \; u \nVdash \check{n} \notin \dot{X}\}$ be the {\em outer hull} of $\dot{X}$ with respect to $u$, i.e. the collection of $n < \omega$ forced to be in $\dot{X}$ by some $u' \leq u$. Note that $u \forces \dot{X} \subseteq \check{X}_u$ for any $u \in \mathbb P$. It follows that for any condition $u$ stronger than $p$, $X_u$ is not in the ideal generated by $\mathscr{I}$. By the maximality of $\mathscr{I}$, moreover we get that for every $r \leq q$ the set $X_r$ is in some complemented filter of $\mathscr{I}$. Therefore to finish the proof it suffices to show that in fact any such $X_r$ is actually in uncountably many such filters. This suffices since if this is the case then in particular it applies to $X_q$ and, since, by the definition of $\mathscr{U}$-encompassing, $Z$ is in $\mathcal{F}(\mathscr{I}, A)$ for co-countably many $A \in \mathscr{I}$ there is some $A \in \mathscr{I}$ so that $Z \cap X_q \in \mathcal{F}(\mathscr{I}, A)$ and so $Z \cap X_q$ has infinite intersection. Thus, some $r \leq q$ forces that $n \in \dot{X}$ for some (in fact infinitely many) $n \in Z$. Summing up, it suffices to show the following claim.

\begin{claim}
For any $u \in \mathbb P$ stronger than $p$ the set $X_u \in \mathcal{F}(\mathscr{I}, A)$  for uncountably many $A \in \mathscr{I}$.
\end{claim}

Fix such a $u \in \mathbb P$ and suppose towards a contradiction that there were only countably many $A \in \mathscr{I}$ with $X_u \in \mathcal{F}(\mathscr{I}, A)$. Let $M \prec H_\theta$ be a countable model for $\theta$ sufficiently large with $\mathbb{P}, p, u, \mathscr{I}, \mathscr{U} \in M$ containing every $A$ so that $X_u \in \mathcal{F}(\mathscr{I}, A)$. Enumerate $\mathscr{I} \cap M$ as $\{A_i \;  | \; i < \omega\}$. Since $p$, and hence $u$, forces that for each $n < \omega$ the name $\dot{X}$ is not in $\mathcal{F}(\mathscr{I}, A_n)$, and each one of such filters is a $p$-point by our assumption and hence an ultrafilter preserved by $\mathbb P$, there is in $M$ a dense set of conditions below $u$ forcing that $\dot{X} \cap A_n \setminus \bigcup_{i < k_n, i \neq n} A_i$ is finite for some $k_n \in \omega$. Applying $\om^\om$-boundedness and properness we can find in the ground model functions $f, g \in \om^\om$ and a condition $r \leq u$ which is $(M, \mathbb{P})$-generic so that for each $n < \omega$ we have 

$$r \forces \dot{X} \cap A_n \setminus \bigcup_{i < \check{f}(n), i \neq n} A_i \subseteq \check{g}(n)$$

In particular we get that $X_r \cap A_n \setminus \bigcup_{i < f(n), i \neq n} A_i \subseteq g(n)$ and thus $X_r \notin \mathcal{F}(\mathscr{I}, A_n)$ for any $n < \omega$. But then, by applying the same argument to $r$ that we applied to $u$, we get that $X_r$ is in some $\mathcal{F}(\mathscr{I}, B)$ for some $B \in \mathscr{I}$ with $B \neq A_n$ for any $n < \omega$. This is a contradiction however since $X_r \subseteq X_u$ and by definition of the $A_n$'s $X_u \notin \mathcal{F}(\mathscr{I}, B)$.  This contradiction implies that $X_q$ is in uncountably many complemented filters of $\mathscr{I}$ and hence the proof is complete.
\end{proof}

As an straightforward corollary we obtain:
\begin{crl}\hfill
\begin{enumerate}
    \item $\mathfrak{s}_{mm} = \aleph_1$ in the Sacks model.
    \item $\mathfrak{s}_{mm} = \aleph_1$ in the Miller partition model and hence $\mathfrak{s}_{mm} < \mathfrak{a}_T$ is consistent.
    \item $\mathfrak{s}_{mm} = \aleph_1$ in the $h$-perfect tree forcing model and hence $\mathfrak{s}_{mm} < \hbox{non}(\mathcal{N})$ 
is consistent.
\end{enumerate}
\end{crl}

\begin{proof}
For (1), it is a standard fact that the iterated Sacks forcing preserves $p$-points and it is $\omega^\omega$-bounding. For (2), in \cite{miller_partition_forcing}, Miller has constructed a forcing, known as Miller partition forcing,  which makes the cardinal invariant $\mathfrak{a}_T$ equal to $\aleph_2$, as recently shown in \cite{JCVFOGJS} preserves $p$-points, and as shown in 
\cite{spinas_partition_miller}) is $\omega^\omega$-bounding. For (3) recall, that the $h$-perfect tree forcing is proper, $^\omega\omega$-bounding, preserves $p$-points and that in the $h$-perfect tree forcing model $\hbox{non}(\mathcal{N})=\aleph_2$, see \cite[Section2]{goldstern-judah-shelah}. 
\end{proof}

An alternation of Miller partition and $h$-perfect tree forcings will lead to a model of $\mathfrak{i}=\mathfrak{s}_{mm}<\hbox{non}(\mathcal{N})=\mathfrak{a}_T=\aleph_2$
(for the effect of the respective posets on $\mathfrak{i}$ see \cite{JCVFOGJS} and \cite{CS}).

\begin{crl}
$\mathfrak{s}_{mm}$ is independent of $\mathfrak{a}_T$
\end{crl}

\begin{proof}
In the Miller partition model, $\mathfrak{s}_{mm} < \mathfrak{a}_T$. On the other hand, it is well known that $\mathfrak{a}_T < \mathfrak{u}$ holds in the Random model and hence $\mathfrak{a}_T < \mathfrak{s}_{mm}$ holds in that model as well.
\end{proof}

\section{Conclusion and Open Questions}

The results of the current paper together with those of \cite{cancino_guzman_miller_2021} give either a $\ZFC$ relation, or establish the independence between $\mathfrak{s}_{mm}$ and any other well studied cardinal characteristic, with the exception of the almost disjointness number $\mathfrak{a}$. The following remains open.

\begin{question}
Is it consistent that $\mathfrak{s}_{mm} < \mathfrak{a}$?
\end{question}

The corresponding question for $\mathfrak{i}$, i.e. the consistency of $\mathfrak{i}<\mathfrak{a}$ is one of the most interesting open problems in cardinal characteristics of the continuum and many of the roadblocks towards solving that problem are the same as trying to answer the question above. See the appendix of \cite{JCVFOGJS} for an interesting discussion on Vaughan's problem. 

As noted in the introduction, Theorem \ref{mainthm3} implies that $\mathfrak{s}_{mm}={\rm max}\{\mathfrak{d}, \mathfrak{u}\}$ in many standard forcing extensions. However, this is not a $\ZFC$ equality as $\mathfrak{s}_{mm}>{\rm max}\{\mathfrak{d}, \mathfrak{u}\}$ holds in the Boolean ultrapower model, see for example \cite{Brendle_templates1}. That model requires a measurable cardinal and increases both $\mathfrak{u}$ and $\mathfrak{d}$. As a result the following two questions remain very interesting:
\begin{question}
Is ${\rm max}\{\mathfrak{d}, \mathfrak{u}\}<\mathfrak{s}_{mm}$ consistent with $\ZFC$?
\end{question}

\begin{question}
If $\mathfrak{d} = \mathfrak{u} = \aleph_1$ does $\mathfrak{s}_{mm} = \aleph_1$?
\end{question}
The later question is an ideal independent version of Roitman's problem. Theorem \ref{mainthm2} opens up the possibility of a maximal ideal independent families of size $\aleph_\omega$. We can therefore ask:

\begin{question}
Is it consistent that $\mathfrak{s}_{mm} = \aleph_\omega$? More generally can $\mathfrak{s}_{mm}$ have countable cofinality?
\end{question}

Finally, we ask more generally about the spectrum of maximal ideal independent families:
\begin{question}
What $\ZFC$ restrictions are there on the set ${\rm spec}(\mathfrak{s}_{mm})$? Can it be equal to any set of regular cardinals which includes the continuum?
\end{question}

\end{document}